\documentclass[11pt]{amsart}
\newcommand{\NN}{\mathrm{I\!N\!}}
\makeatletter
\def\subsection{\@startsection{subsection}{3}
  \z@{.5\linespacing\@plus.7\linespacing}{.1\linespacing}
  {\normalfont\itshape}}
\makeatother

\usepackage{amssymb}
\usepackage{amsmath}
\usepackage{mathrsfs}
\usepackage{amsthm}
\usepackage{enumerate}
\newtheorem{theorem}{Theorem}[section]

\newtheorem{corollary}[theorem]{Corollary}
\newtheorem{proposition}[theorem]{Proposition}
\theoremstyle{definition}

\newtheorem*{remark}{Remark}
\newtheorem*{remarks}{Remarks}
\numberwithin{equation}{section}
\mathchardef\hyphen="2D
\def\@tvsp{\mathchoice{{}\mkern-4.5mu}{{}\mkern-4.5mu}{{}\mkern-2.5mu}{}}
\def\ln{\left|\@tvsp\left|\@tvsp\left|}
\def\rn{\right|\@tvsp\right|\@tvsp\right|}
\makeatother
\begin{document}
\title{A note on ``Weak limits of almost invariant projections''}
\author{March T.~Boedihardjo}
\address{Department of Mathematics, Texas A\&M University, College Station, Texas 77843}
\email{march@math.tamu.edu}
\keywords{}
\subjclass[2010]{}
\begin{abstract}
We give alternative proofs to certain results in the paper \cite{Foias} by using ultraproducts of operators.
\end{abstract}
\maketitle
\allowdisplaybreaks
\section{Introduction}\label{1}
In \cite{Boedihardjo}, the author obtained some results concerning operators on Hilbert space inspired by ultraproducts and finite representability of Banach spaces. In
particular, the author obtained an alternative proof of Hadwin's characterization of the WOT, SOT and $*$-SOT closure of the unitary orbit of a given operator on
Hilbert space \cite[Theorem 4.4]{Hadwin1} and an affirmative answer to a question of Hadwin \cite[Question 9]{Hadwin2}. The key ingredients in the proofs of these are
ultraproducts of operators, the Calkin representation and \cite[Theorem 1.3]{Voiculescu}. The purpose of this paper is to extend the use of these ingredients to give
alternative proofs to certain results in \cite{Foias}.

Let $\mathcal{H}$ be a separable, infinite dimensional, complex Hilbert space. The algebra of bounded linear operators on $\mathcal{H}$ is denoted by $\mathcal{B(H)}$, and
the ideal of compact operators in $\mathcal{B(H)}$ is denoted by $\mathcal{K(H)}$. Let $p$ be the quotient map from $\mathcal{B(H)}$ onto $\mathcal{B(H)}/\mathcal{K(H)}$.

An algebra $\mathcal{A}\subset\mathcal{B(H)}$ is {\it reductive} if every subspace of $\mathcal{H}$ invariant under $\mathcal{A}$ reduces $\mathcal{A}$; $\mathcal{A}$ is
{\it strongly reductive} (see \cite{Harrison} and \cite{Apostol}) if for every sequence $(P_{n})_{n=1}^{\infty}$ of projections in $\mathcal{B(H)}$
satisfying
\[\lim_{n\to\infty}\|(I-P_{n})TP_{n}\|=0,\quad T\in\mathcal{A},\]
we have
\[\lim_{n\to\infty}\|TP_{n}-P_{n}T\|=0,\quad T\in\mathcal{A}.\]
Let $\psi_{1},\psi_{2}:\mathcal{A}\to\mathcal{B(H)}$ be two representations of an algebra $\mathcal{A}\subset\mathcal{B(H)}$. ``We say that $\psi_{2}$ is in the
norm-closed unitary orbit of $\psi_{1}$, if there exists a sequence $(U_{n})_{n=1}^{\infty}$ of unitary operators such that:
\[\lim_{n\to\infty}\|\psi_{2}(T)-U_{n}\psi_{1}(T)U_{n}^{-1}\|=0,\]
for all $T\in\mathcal{A}$."

The following characterization of strongly reductive separable commutative algebras was obtained in \cite{Foias}.
\begin{theorem}\label{12}
Let $\mathcal{A}\subset\mathcal{B(H)}$ be a norm-separable commutative algebra containing $I$. The following properties are equivalent:
\begin{enumerate}[(i)]
\item $\mathcal{A}$ is strongly reductive,
\item the norm-closure of $\mathcal{A}$ is a $C^{*}$-algebra,
\item for every representation $\rho$ of $\mathcal{A}$ in the norm-closed unitary orbit of the identity representation of $\mathcal{A}$ on $\mathcal{H}$, the algebra
$\rho(\mathcal{A})$ is reductive.
\end{enumerate}
\end{theorem}
Note that (ii)$\Rightarrow$(i) is obvious, and (i)$\Rightarrow$(iii) is simple and elementary but slightly technical (see \cite[page 92]{Foias}). The main part of Theorem
\ref{12} is (iii)$\Rightarrow$(ii). It was asked at the end of \cite{Foias} whether there is a simple, direct proof of (iii)$\Rightarrow$(i).

In Section 2, we recall the ultraproducts of operators, the Calkin representation and \cite[Theorem 1.3]{Voiculescu} which are needed in the rest of this paper. Section 2
is essentially the same as the beginning part of \cite[Section 3]{Boedihardjo}.

In Section 3, we give a direct proof of (iii)$\Rightarrow$(i) in Theorem \ref{12}. From this proof together with a result in \cite{Prunaru}, we obtain that Theorem
\ref{12} is true without the commutativity assumption on $\mathcal{A}$.

In Section 4, we give an alternative proof of an implication in the main result of \cite{Foias}, which asserts the equivalence of three statements about a separable closed
algebra $\mathcal{A}$ in $\mathcal{B(H)}$ containing $I$ and a positive contraction $Q\in\mathcal{B(H)}$. From this alternative proof, we obtain a slight improvement of
the main result of \cite{Foias}.
\section{Ultraproducts of operators}\label{2}
Let $\mathscr{U}$ be a free ultrafilter on $\NN\,$. If $(a_{n})_{n\geq 1}$ is a bounded sequence in $\mathbb{C}$, then its ultralimit through $\mathscr{U}$ is denoted by
$\displaystyle\lim_{n,\mathscr{U}}a_{n}$. Consider the Banach space
\[\mathcal{H}^{\mathscr{U}}:=l^{\infty}(\mathcal{H})/\left\{(x_{n})_{n\geq 1}\in l^{\infty}(\mathcal{H}):\lim_{n,\mathscr{U}}\|x_{n}\|=0\right\}.\]
If $(x_{n})_{n\geq 1}\in l^{\infty}(\mathcal{H})$ then its image in $\mathcal{H}^{\mathscr{U}}$ is denoted by $(x_{n})_{\mathscr{U}}$, and it can be easily checked that
\[\|(x_{n})_{\mathscr{U}}\|=\lim_{n,\mathscr{U}}\|x_{n}\|.\]
Moreover, $\mathcal{H}^{\mathscr{U}}$ is, in fact, a Hilbert space with inner product
\[\langle(x_{n})_{\mathscr{U}},(y_{n})_{\mathscr{U}}\rangle=\lim_{n,\mathscr{U}}\langle x_{n},y_{n}\rangle.\]
But $\mathcal{H}^{\mathscr{U}}$ is nonseparable.

If $(T_{n})_{n\geq 1}$ is a bounded sequence in $\mathcal{B(H)}$, then its {\it ultraproduct} $(T_{1},T_{2},\ldots)_{\mathscr{U}}\in\mathcal{B}(\mathcal{H}^{\mathscr{U}})$
is defined by $(x_{n})_{\mathscr{U}}\mapsto (T_{n}x_{n})_{\mathscr{U}}$. If $T\in\mathcal{B(H)}$ then its {\it ultrapower} $T^{\mathscr{U}}\in\mathcal{B}(\mathcal{H}^{
\mathscr{U}})$ is defined by $(x_{n})_{\mathscr{U}}\mapsto (Tx_{n})_{\mathscr{U}}$. It is easy to see that
\[\|(T_{1},T_{2},\ldots)_{\mathscr{U}}\|=\lim_{n,\mathscr{U}}\|T_{n}\|,\]
\[(T_{1},T_{2},\ldots)_{\mathscr{U}}^{*}=(T_{1}^{*},T_{2}^{*},\ldots)_{\mathscr{U}},\]
and in particular, $(T^{\mathscr{U}})^{*}=(T^{*})^{\mathscr{U}}$.

Consider the subspace
\[\widehat{\mathcal{H}}:=\left\{(x_{n})_{\mathscr{U}}\in \mathcal{H}^{\mathscr{U}}:w\hyphen\lim_{n,\mathscr{U}}x_{n}=0\right\}.\]
Here $\displaystyle w\hyphen\lim_{n,\mathscr{U}}x_{n}$ is the weak limit of $(x_{n})_{n\geq 1}$ through $\mathscr{U}$, i.e., the unique element $x\in\mathcal{H}$ such that
\begin{equation}\label{21e}
\langle x,y\rangle=\lim_{n,\mathscr{U}}\langle x_{n},y\rangle,\quad y\in\mathcal{H}.
\end{equation}
Consider also the (closed) subspace $\{(x)_{\mathscr{U}}:x\in\mathcal{H}\}$ of $\mathcal{H}^{\mathscr{U}}$. The projection from $\mathcal{H}^{\mathscr{U}}$ onto this
subspace is given by $\displaystyle(x_{n})_{\mathscr{U}}\mapsto(w\hyphen\lim_{k,\mathscr{U}}x_{k})_{\mathscr{U}}$, and so $\{(x)_{\mathscr{U}}:x\in\mathcal{H}\}^{\perp}
=\widehat{\mathcal{H}}$. We shall identify $\{(x)_{\mathscr{U}}:x\in\mathcal{H}\}$ with $\mathcal{H}$. So we have $\mathcal{H}^{\mathscr{U}}=\mathcal{H}\oplus
\widehat{\mathcal{H}}$.

For $T\in\mathcal{B(H)}$, $\mathcal{\widehat{H}}$ is a reducing subspace for $T^{\mathscr{U}}$ and define $\widehat{T}\in\mathcal{B}(\mathcal{\widehat{H}})$ by
\[\widehat{T}:=T^{\mathscr{U}}|_{\mathcal{\widehat{H}}}.\]
Thus, we have
\begin{equation}\label{22e}
T^{\mathscr{U}}=T\oplus\widehat{T}
\end{equation}
with respect to the decomposition $\mathcal{H}^{\mathscr{U}}=\mathcal{H}\oplus\widehat{\mathcal{H}}$.

Note that $\widehat{K}=0$ for $K\in\mathcal{K(H)}$. (The proof of this  uses the topological definition of weak ultralimit rather than (\ref{21e}) above and uses also the
fact that every sequence in a compact metric space converges to an element through $\mathscr{U}$. This compact Hausdorff space is taken to be the norm closure of the image
of the unit ball of $\mathcal{H}$ under $K$ equipped with the norm topology.)
The map $f:\mathcal{B(H)}/\mathcal{K(H)}\to\mathcal{B}(\widehat{\mathcal{H}})$ defined by $p(T)\mapsto\widehat{T}$ is the {\it Calkin representation}.
\begin{theorem}[\cite{Calkin}, Theorem 5.5]\label{21}
The map $f$ is an isometric $*$-isomorphism into $\mathcal{B}(\mathcal{\widehat{H}})$.
\end{theorem}
Let us recall the definition of approximate unitary equivalence of representations and a result of Voiculescu.

Let $\psi_{1},\psi_{2}:\mathcal{A}\to\mathcal{B(H)}$ be two representations of an algebra $\mathcal{A}\subset\mathcal{B(H)}$. Then $\psi_{1}$ and $\psi_{2}$ are
{\it approximately unitarily equivalent} \cite{Voiculescu}, denoted by $\psi_{1}\sim_{a}\psi_{2}$, if there is a sequence $(U_{n})_{n=1}^{\infty}$ of unitary operators
such that
\[\psi_{2}(T)-U_{n}\psi_{1}(T)U_{n}^{-1}\in\mathcal{K(H)},\quad n\geq 1,\]
and
\[\lim_{n\to\infty}\|\psi_{2}(T)-U_{n}\psi_{1}(T)U_{n}^{-1}\|=0\]
for all $T\in\mathcal{A}$. Note that if $\psi_{1}\sim_{a}\psi_{2}$ then $\psi_{2}$ is in the norm-closed unitary orbit of $\psi_{1}$.
\begin{theorem}[\cite{Voiculescu}, Theorem 1.3]\label{22}
Let $\mathcal{A}$ be a separable $C^{*}$-algebra with unit and $\rho$ a representation of $\mathcal{A}$ on $\mathcal{H}$. Let $\pi$ be a representation of $p(\rho(
\mathcal{A}))$ on a separable Hilbert space $\mathcal{H}_{\pi}$. Then $\rho\sim_{a}\rho\oplus\pi\circ p\circ\rho$.
\end{theorem}
Suppose now that $\mathcal{A}\subset\mathcal{B(H)}$. Take $\rho$ to be the identity representation $\mathrm{id}$ of $\mathcal{A}$ on $\mathcal{H}$. If $\mathcal{M}$ is a
separable subspace of $\widehat{\mathcal{H}}$ that reduces $(f\circ p)(\mathcal{A})$, then $f|_{\mathcal{M}}$ defines a representation of $p(\mathcal{A})$ on
$\mathcal{M}$. Taking $\pi$ to be this representation in Theorem \ref{22} and $\mathcal{H}_{\pi}=\mathcal{M}$, we obtain
\begin{corollary}\label{23}
Let $\mathcal{A}$ be a separable $C^{*}$-subalgebra of $\mathcal{B(H)}$ containing $I$. Let $\mathcal{M}$ be a separable subspace of $\widehat{\mathcal{H}}$ that reduces
$(f\circ p)(\mathcal{A})$. Then $\mathrm{id}\sim_{a}\mathrm{id}\oplus[(f\circ p\circ\mathrm{id})|_{\mathcal{M}}]$.
\end{corollary}
\section{Proof of (iii)$\Rightarrow$(i) in Theorem \ref{12}}\label{4}
\begin{proposition}\label{41}
If (iii) in Theorem \ref{12} holds then the algebra $\{T^{\mathscr{U}}:T\in\mathcal{A}\}$ in $\mathcal{B}(\mathcal{H}^{\mathscr{U}})$ is reductive.
\end{proposition}
\begin{proof}
By Corollary \ref{23}, for every separable reducing subspace $\mathcal{M}$ of $\widehat{H}$ that reduces $(f\circ p)(\mathcal{A})$, we have $\mathrm{id}\sim_{a}\mathrm{id}
\oplus[(f\circ p\circ\mathrm{id})|_{\mathcal{M}}]$, and so by assumption,
\[(\mathrm{id}\oplus[(f\circ p\circ\mathrm{id})|_{\mathcal{M}}])(\mathcal{A})=\{T\oplus [f(p(T))|_{\mathcal{M}}]:T\in\mathcal{A}\}\]
is reductive. But $T^{\mathscr{U}}|_{\mathcal{H}\oplus\mathcal{M}}=T\oplus(\widehat{T}|_{\mathcal{M}})=T\oplus [f(p(T))|_{\mathcal{M}}]$. Therefore,
$\{T^{\mathscr{U}}|_{\mathcal{H}\oplus\mathcal{M}}:T\in\mathcal{A}\}$ is reductive.

For every separable subspace $\mathcal{N}$ of $\mathcal{H}^{\mathscr{U}}$, there is a separable reducing subspace $\mathcal{M}$ for $(f\circ p)(\mathcal{A})$ such that
$\mathcal{N}\subset\mathcal{H}\oplus\mathcal{M}$. (Take, for example, $\mathcal{M}$ to be the smallest subspace of $\widehat{\mathcal{H}}$ that contains $P_{\widehat{H}}
\mathcal{N}$ and reduces $(f\circ p)(\mathcal{A})$.) Thus, if $\mathcal{N}$ is invariant under $\{T^{\mathscr{U}}:T\in\mathcal{A}\}$ then $\mathcal{N}$ is invariant under
$\{T^{\mathscr{U}}|_{\mathcal{H}\oplus\mathcal{M}}:T\in\mathcal{A}\}$. Since $\{T^{\mathscr{U}}|_{\mathcal{H}\oplus\mathcal{M}}:T\in\mathcal{A}\}$ is reductive, this
implies that $\mathcal{N}$ reduces $\{T^{\mathscr{U}}|_{\mathcal{H}\oplus\mathcal{M}}:T\in\mathcal{A}\}$ and thus reduces $\{T^{\mathscr{U}}:T\in\mathcal{A}\}$. Therefore,
every separable subspace of $\mathcal{H}^{\mathscr{U}}$ that is invariant under $\{T^{\mathscr{U}}:T\in\mathcal{A}\}$ reduces $\{T^{\mathscr{U}}:T\in\mathcal{A}\}$.

Suppose now that $\mathcal{N}$ is a subspace of $\mathcal{H}^{\mathscr{U}}$ invariant under $\{T^{\mathscr{U}}:T\in\mathcal{A}\}$ but $\mathcal{N}$ is not necessarily
separable. Let $z\in\mathcal{N}$. Then $\{T^{\mathscr{U}}z:T\in\mathcal{A}\}$ is a separable subspace of $\mathcal{H}^{\mathscr{U}}$ that is invariant under
$\{T^{\mathscr{U}}:T\in\mathcal{A}\}$. So by the conclusion of the previous paragraph, $\{T^{\mathscr{U}}z:T\in\mathcal{A}\}$ reduces $\{T^{\mathscr{U}}:T\in\mathcal{A}
\}$. Thus, $(T^{\mathscr{U}})^{*}z\in\{T^{\mathscr{U}}z:T\in\mathcal{A}\}$ for all $T\in\mathcal{A}$. Since $\mathcal{N}$ is invariant under $\{T^{\mathscr{U}}:T\in
\mathcal{A}\}$, this implies that $(T^{\mathscr{U}})^{*}z\in\mathcal{N}$ for all $T\in\mathcal{A}$ and $z\in\mathcal{N}$. Therefore, $\mathcal{N}$ reduces $\{T^{\mathscr{U
}}:T\in\mathcal{A}\}$. It follows that $\{T^{\mathscr{U}}:T\in\mathcal{A}\}$ is reductive.
\end{proof}
We are now ready to complete the proof of (iii)$\Rightarrow$(i) in Theorem \ref{12}. Suppose that (iii) is true and (i) is not true. Then there exist $\epsilon>0$, $T_{0}
\in\mathcal{A}$ and a sequence $(P_{n})_{n\geq}$ of projections in $\mathcal{B(H)}$ such that $\displaystyle\lim_{n\to\infty}\|(I-P_{n})TP_{n}\|=0$ for all $T\in
\mathcal{A}$ but $\|T_{0}P_{n}-P_{n}T_{0}\|\geq\epsilon$ for all $n\geq 1$.

Note that $(P_{1},P_{2},\ldots)_{\mathscr{U}}$ is a projection in $\mathcal{B}(\mathcal{H}^{\mathscr{U}})$ and
\[(I-(P_{1},P_{2},\ldots)_{\mathscr{U}})T^{\mathscr{U}}(P_{1},P_{2},\ldots)_{\mathscr{U}}=((I-P_{1})TP_{1},(I-P_{2})TP_{2},\ldots)_{\mathscr{U}}=0\]
for all $T\in\mathcal{A}$. So by Proposition \ref{41}, $T^{\mathscr{U}}(P_{1},P_{2},\ldots)_{\mathscr{U}}=(P_{1},P_{2},\ldots)_{\mathscr{U}}T^{\mathscr{U}}$ for all $T\in
\mathcal{A}$. This means that
\[\lim_{n,\mathscr{U}}\|TP_{n}-P_{n}T\|=0,\quad T\in\mathcal{A}.\]
But $\|T_{0}P_{n}-P_{n}T_{0}\|\geq\epsilon$ for all $n\geq 1$ which is a contradiction. Therefore, (iii)$\Rightarrow$(i).
\begin{remarks}
The proof given in this section does not use the commutativity of $\mathcal{A}$. So (iii)$\Rightarrow$(i) in Theorem \ref{12} holds without the commutativity of
$\mathcal{A}$. Moreover, it was shown in \cite{Foias} that the converse direction (i)$\Rightarrow$(iii) also does not require the commutativity of $\mathcal{A}$ (nor the
separability).

In \cite{Prunaru}, it was proved that the norm closure of a strongly reductive algebra is self-adjoint which means that (i)$\Rightarrow$(ii) is true without the
commutativity and the separability of $\mathcal{A}$. Since (ii)$\Rightarrow$(i) is obviously true also without these conditions on $\mathcal{A}$, it follows that
(i)$\Leftrightarrow$(ii) is true without the commutativity and the separability of $\mathcal{A}$.

We conclude that Theorem \ref{12} holds without the commutativity of $\mathcal{A}$, whereas the separability of $\mathcal{A}$ is only needed for the implications
(iii)$\Rightarrow$(i) and (iii)$\Rightarrow$(ii).
\end{remarks}
\section{Main result of \cite{Foias}}\label{3}
The main result of \cite{Foias} is
\begin{theorem}\label{11}
Let $\mathcal{A}\subset\mathcal{B(H)}$ be a norm-separable norm closed algebra containing $I$, and $Q\in\mathcal{B(H)}$, $0\leq Q\leq I$. Then the following statements are
equivalent.
\begin{enumerate}[(i)]
\item There exists a sequence $(P_{n})_{n=1}^{\infty}$ of projections in $\mathcal{B(H)}$ such that $\displaystyle\lim_{n\to\infty}\|(I-P_{n})TP_{n}\|=0$ for all
$T\in\mathcal{A}$ and $\displaystyle w\hyphen\lim_{n\to\infty}P_{n}=Q$.
\item There exists a sequence $(R_{n})_{n=1}^{\infty}$ of projections in $\mathcal{B(H)}$ such that $\displaystyle w\hyphen\lim_{n\to\infty}(I-R_{n})TR_{n}=0$ for all $T
\in\mathcal{A}$ and $\displaystyle w\hyphen\lim_{n\to\infty}R_{n}=Q$.
\item There exists a representation $\rho$ of $p(C^{*}(\mathcal{A}))$ on some separable Hilbert space $\mathcal{H}'$ and a subspace $L\subset\mathcal{H}\oplus\mathcal{H}'$
invariant under $(\mathrm{id}\oplus(\rho\oplus p))(\mathcal{A})$ such that
\[P_{\mathcal{H}\oplus 0}P_{L}|_{\mathcal{H}\oplus 0}=Q.\]
\end{enumerate}
\end{theorem}
Note that (i)$\Rightarrow$(ii) is trivial and both (ii)$\Rightarrow$(iii) and (iii)$\Rightarrow$(i) are nontrvial. In this section, we give an alternative proof of
(ii)$\Rightarrow$(iii). We start the proof with the following proposition.
\begin{proposition}\label{31}
Let $\mathcal{A}$ be a norm-separable algebra containing $I$ and let $Q\in\mathcal{B(H)}$. If there exists a bounded sequence $(R_{n})_{n=1}^{\infty}$ in $\mathcal{B(H)}$
such that $\displaystyle w\hyphen\lim_{n\to\infty}(I-R_{n}^{*})TR_{n}=0$ for all $T\in\mathcal{A}$ and $\displaystyle w\hyphen\lim_{n\to\infty}R_{n}=Q$, then there is a
separable subspace $L$ of $\mathcal{H}^{\mathscr{U}}$ invariant under $\{T^{\mathscr{U}}:T\in\mathcal{A}\}$ such that
\[P_{\mathcal{H}\oplus 0}P_{L}|_{\mathcal{H}\oplus 0}=Q.\]
\end{proposition}
\begin{proof}
Take
\[L=\vee\{(TR_{n}y)_{\mathscr{U}}:T\in\mathcal{A},\,y\in\mathcal{H}\}.\]
Then $L$ is a separable subspace of $\mathcal{H}^{\mathscr{U}}$ that is invariant under $T^{\mathscr{U}}$ for every $T\in\mathcal{A}$. It remains to show that
\[P_{\mathcal{H}\oplus 0}P_{L}|_{\mathcal{H}\oplus 0}=Q.\]
For every $x,y\in\mathcal{H}$,
\begin{eqnarray*}
\langle(x)_{\mathscr{U}}-(R_{n}x)_{\mathscr{U}},(TR_{n}y)_{\mathscr{U}}\rangle&=&\langle((I-R_{n})x)_{\mathscr{U}},(TR_{n}y)_{\mathscr{U}}\rangle
\\&=&\lim_{n,\mathscr{U}}\langle (I-R_{n})x,TR_{n}y\rangle\\&=&\lim_{n,\mathscr{U}}\langle x,(I-R_{n}^{*})TR_{n}y\rangle=0\text{ by assumption}.
\end{eqnarray*}
Thus, $((x)_{\mathscr{U}}-(R_{n}x)_{\mathscr{U}})\perp L$ for every $x\in\mathcal{H}$. But $(R_{n}x)_{\mathscr{U}}\in L$. Therefore, by the definition of orthogonal
projection onto $L$,
\[P_{L}(x)_{\mathscr{U}}=(R_{n}x)_{\mathscr{U}}.\]
Taking $P_{\mathcal{H}\oplus 0}$ on both sides, we obtain
\[P_{\mathcal{H}\oplus 0}P_{L}(x)_{\mathscr{U}}=P_{\mathcal{H}\oplus 0}(R_{n}x)_{\mathscr{U}}=w\hyphen\lim_{n,\mathscr{U}}R_{n}x=Qx.\]
\end{proof}
We are now ready to complete the proof of (ii)$\Rightarrow$(iii) in Theorem \ref{11}.

Assume (ii). Applying Proposition \ref{31}, we obtain a separable subspace $L$ of $\mathcal{H}^{\mathscr{U}}$ invariant under $\{T^{\mathscr{U}}:T\in\mathcal{A}\}$
such that
\[P_{\mathcal{H}\oplus 0}P_{L}|_{\mathcal{H}\oplus 0}=Q.\]
By (\ref{22e}),
\[T^{\mathscr{U}}=T\oplus\widehat{T}=T\oplus f(p(T))=(\mathrm{id}\oplus(f\circ p))(T).\]
Take $\mathcal{H}'$ to be the smallest subspace of $\widehat{\mathcal{H}}$ that contains $P_{\mathcal{\widehat{H}}}L$ and reduces $(f\circ p)(C^{*}(\mathcal{A}))$. Note
that $\mathcal{H}'$ is separable. Take $\rho$ to be $S\mapsto f(S)|_{\mathcal{H}'}$ for $S\in p(C^{*}(\mathcal{A}))$. We obtain (iii).
\begin{remark}
Since the assumption of Proposition \ref{31} is slightly weaker than (ii) in Theorem \ref{11}, we have the following slight improvement of Theorem \ref{11}.
\begin{theorem}
Let $\mathcal{A}$ be a norm-separable norm closed algebra containing $I$, and $Q\in\mathcal{B(H)}$, $0\leq Q\leq I$. Then the following statements are equivalent.
\begin{enumerate}[(i)]
\item There exists a sequence $(P_{n})_{n=1}^{\infty}$ of projections in $\mathcal{B(H)}$ such that $\displaystyle\lim_{n\to\infty}\|(I-P_{n})TP_{n}\|=0$ for all
$T\in\mathcal{A}$ and $\displaystyle w\hyphen\lim_{n\to\infty}P_{n}=Q$.
\item There exists a bounded sequence $(R_{n})_{n=1}^{\infty}$ in $\mathcal{B(H)}$ such that $\displaystyle w\hyphen\lim_{n\to\infty}(I-R_{n}^{*})TR_{n}=0$ for all $T
\in\mathcal{A}$ and $\displaystyle w\hyphen\lim_{n\to\infty}R_{n}=Q$.
\item There exists a representation $\rho$ of $p(C^{*}(\mathcal{A}))$ on some separable Hilbert space $\mathcal{H}'$ and a subspace $L\subset\mathcal{H}\oplus\mathcal{H}'$
invariant under $(\mathrm{id}\oplus(\rho\oplus p))(\mathcal{A})$ such that
\[P_{\mathcal{H}\oplus 0}P_{L}|_{\mathcal{H}\oplus 0}=Q.\]
\end{enumerate}
\end{theorem}
\end{remark}

\end{document}